\documentclass[11pt,a4paper,english,reqno]{amsart}
\usepackage{amsmath,amssymb,amsfonts,epsfig,mathrsfs}
\usepackage[T1]{fontenc}

\usepackage{color}
\usepackage{array}
\usepackage{amsthm}
\usepackage{amstext}
\usepackage{graphicx}
\usepackage{setspace}
\usepackage[margin=2.5cm]{geometry}
\usepackage{bbm}
\usepackage{color}
\usepackage{enumitem}
\usepackage{undertilde}
\allowdisplaybreaks[4]

\usepackage{amscd,psfrag}
\usepackage{yhmath}
\usepackage[mathscr]{eucal}

\usepackage{slashed}

\makeatletter
\pdfpageheight\paperheight
\pdfpagewidth\paperwidth

\usepackage{epstopdf}
\usepackage{chngcntr}
\counterwithin{figure}{section}
\usepackage{mathrsfs}

\usepackage{indentfirst}	

\usepackage[normalem]{ulem}
\theoremstyle{plain}

\newtheorem{definition}{Definition}[section]
\newtheorem{theorem}[definition]{Theorem}
\newtheorem*{theorem*}{Theorem}

\newtheorem*{remark*}{Remark}
\newtheorem*{sideremark*}{Side Remark}\newtheorem*{mt*}{Main Theorem}

\newtheorem*{claim*}{Claim}
\newtheorem*{q*}{Question}
\newtheorem{lemma}[definition]{Lemma}

\newtheorem*{corollary*}{Corollary}
\newtheorem*{proposition*}{Proposition}

\newcommand{\R}{\mathbb{R}}

\newcommand{\na}{\nabla}

\newcommand{\dd}{{\rm d}}
\newcommand{\p}{\partial}

\newcommand{\map}{\rightarrow}
\newcommand{\G}{\Gamma}

\newcommand{\xyhat}{\widehat{x-y}}
\newcommand{\uu}{{u \otimes u}}

\newcommand{\dv}{{\rm div}}
\newcommand{\obig}{\omega^{(>)}}
\newcommand{\osmall}{\omega^{(<)}}

\newcommand{\scal}{\mathcal{S}}
\newcommand{\what}{\widehat{\omega}}
\newcommand{\sbig}{{\scal}^{(>)}}
\newcommand{\ssmall}{{\scal}^{(<)}}
\newcommand{\wbig}{{\what}^{(>)}}

\allowdisplaybreaks[4]

\def\XXint#1#2#3{{\setbox0=\hbox{$#1{#2#3}{\int}$ }
\vcenter{\hbox{$#2#3$ }}\kern-.6\wd0}}

\numberwithin{equation}{section}
\numberwithin{figure}{section}

\title{On vortex alignment and boundedness of \\$L^q$ norm of vorticity}

\author{Siran Li}
\address{Siran Li: Department of Mathematics, Rice University, MS 136
P.O. Box 1892, Houston, Texas, 77251-1892, USA; \, $\bullet$ \,  Department of Mathematics, McGill University, Burnside Hall, 805 Sherbrooke Street West, Montreal, Quebec, H3A 0B9, Canada;\, $\bullet$ \,  Centre de Recherches Math\'{e}matiques, Universit\'{e} de Montr\'{e}al, P.O. Box 6128, Centre-ville Station. Montr\'{e}al, Qu\'{e}bec, H3C 3J7, Canada.}

\email{\texttt{Siran.Li@rice.edu}}

\date{\today}

\begin{document}

\maketitle

\begin{abstract}
We show that the spatial $L^q$ ($q>5/3$) norm of the vorticity of an incompressible viscous fluid in $\R^3$ or $\mathbb{T}^3$ remains bounded uniformly in time, provided that the direction of vorticity is H\"{o}lder continuous in the space variable, and that the space--time $L^q$ norm of the vorticity is finite. The H\"{o}lder index depends only on $q$. This serves as a variant of the classical result by P. Constantin and Ch. Fefferman (Direction of vorticity and the problem of global regularity for the Navier--Stokes equations, \textit{Indiana Univ. J. Math.}, \textbf{42} (1993), 775--789).

\end{abstract}

\section{Introduction}
In this paper we consider the Cauchy problem of incompressible Navier--Stokes equations:
\begin{eqnarray}
&\p_t u + \dv\,\uu - \nu \Delta u +\na p=0 \qquad \text{ in } ]0,T]\times \Omega,\label{NS eqn}\\
&\dv\, u = 0 \qquad \text{ in } ]0,T]\times \Omega,\label{incompressibility}\\
&u|_{t=0}=u_0 \qquad \text{ on } \{0\}\times \Omega,\label{initial data}
\end{eqnarray}
where $\Omega = \R^3$ or $\mathbb{T}^3$. The constant $\nu>0$ is the viscosity, $u:\Omega\map\R^3$ the velocity, and $p:\Omega\map\R$ the pressure of the fluid. The existence, uniqueness and regularity of Eqs.\,\eqref{NS eqn}--\eqref{initial data} has been  a central research topic of nonlinear PDEs; see Fefferman \cite{f}, Constantin--Foias \cite{cfs}, Seregin \cite{s} and many other references.

The vorticity $\omega:=\na\times u$ is an important quantity for the fluid motion. Its time evolution is determined by the {\em vorticity equation}, which can be obtained by taking the curl of Eq.\,\eqref{NS eqn}:
\begin{equation}
\p_t \omega + (u\cdot \na)\omega -\nu\Delta\omega = \scal \cdot \omega,
\end{equation} 
where $\scal$ is the $3\times 3$ matrix
\begin{equation}
\scal := \frac{\na u + \na^\top u}{2}.
\end{equation}
The alignment of the vorticity is closely related to the regularity of the weak solutions to the Navier--Stokes equations. A celebrated result by Constantin--Fefferman (\cite{cf}) shows that, if the vorticity direction does not change too rapidly in the regions  with high vorticity magnitude, then a weak solution is automatically strong.  More precisely, denote by
\begin{equation}\label{angle}
\varphi(t,x,y):=\angle \,(\omega(t,x), \omega(t,y)).
\end{equation} 
If there exist $\Lambda, \rho>0$ such that 
\begin{equation}\label{cf condition}
|\sin \varphi(t,x,y)| \leq \frac{|x-y|}{\rho}
\end{equation}
whenever $|\omega(t,x)|, |\omega(t,y)| \geq \Lambda$, then a weak solution $u$ on $[0,T]$ must be a classical solution on $[0,T]$. Here, weak solutions are defined in the Leray--Hopf sense: $u \in L^\infty(0,T; L^2(\Omega)) \cap L^2(0,T; H^1(\Omega))$ with the energy inequality
\begin{equation}\label{energy inequality}
\frac{1}{2}\int |u(t,x)|^2\,\dd x + \nu \int_0^t \int |\na u(s,x)|^2\,\dd s\dd x \leq \frac{1}{2} \int |u_0(t)|^2\,\dd x.
\end{equation}
Throughout the paper, $\int$ without subscripts denotes the integration over $\Omega$, and $\|\cdot\|_{L^q}\equiv \|\cdot\|_{L^q(\Omega)}$.

The above result (Constantin--Fefferman \cite{cf}) is established by showing 
\begin{equation}\label{vorticity growth}
\omega \in L^\infty\big(0,T; L^2(\Omega)\big) \cap L^2(0,T; H^1\big(\Omega)\big),
\end{equation}
which together with Eq.\,\eqref{incompressibility} implies that $u$ is classical. Using more refined estimates, Beir\~{a}o da Veiga and Berselli improved the Lipschitz condition \eqref{cf condition} in \cite{cf} to a H\"{o}lder condition:
\begin{equation}\label{holder assumption}
|\sin \varphi(t,x,y)| \leq \frac{|x-y|^\beta}{\rho}, \qquad \text{ where } \beta \in \big[\frac{1}{2},1\big].
\end{equation}
The H\"{o}lder exponent $\beta =1/2$ is the best up to date. There is an extensive literature on the geometric regularity conditions {\it \`{a} la} Constantin--Fefferman; see Beir\~{a}o da Veiga--Berselli \cite{bb1, bb2}, Beir\~{a}o da Veiga \cite{b1, b2, b3}, Berselli \cite{berselli}, Chae \cite{chae}, Chae--Kang--Li \cite{ckl}, Giga--Miura \cite{gm}, Gruji\'{c} \cite{grujic}, Vasseur \cite{vasseur} and Zhou \cite{zhou}, as well as the references cited therein. Similar conditions for the Euler equations have also been studied; {\it cf.} Constantin--Fefferman--Majda \cite{cfm}.

This paper serves as a variant of the above results in \cite{cf, bb1}. In comparison with Eq.\,\eqref{vorticity growth} concerning the growth of the $L^2$ norm of vorticity $\omega$, we shall study the growth of the $L^q$ norm of $\omega$ under assumptions of the form Eq.\,\eqref{holder assumption}, in which the H\"{o}lder exponent depends on $q$. More precisely, the main result of the paper is as follows:
\begin{theorem}\label{thm: main}
Let $u:\Omega \times [0,T] \map \R^3$ be a weak solution to Eqs.\,\eqref{NS eqn}--\eqref{initial data}, $\Omega = \R^3$ or $\mathbb{T}^3$. Assume that, for $q > 5/3$, there exist $\Lambda, \rho>0$ such that 
\begin{equation}
|\sin \varphi(t,x,y)| \leq \frac{|x-y|^\beta}{\rho} \qquad \text{ where } \beta \in \Big]\max\Big\{0, \frac{5}{q}-2\Big\}, 1\Big]
\end{equation}
whenever $|\omega(t,x)|, |\omega(t,y)| \geq \Lambda$; the angle $\varphi$ is as in Eq.\,\eqref{angle}. In addition, suppose that $\omega\in L^q(\Omega\times [0,T])$. Then
\begin{equation}
\omega \in L^\infty\big(0,T; L^q(\Omega)\big) \qquad \text{ and } \qquad |\omega|^{q/2} \in L^2\big(0,T; H^1(\Omega)\big).
\end{equation}
\end{theorem}
In particular, for $q=2$, $\beta=1$ Theorem \ref{thm: main} recovers the result by Constantin--Fefferman \cite{cf}; and for $q=2$, $\beta=1/2$ the result by Beir\~{a}o da Veiga--Berselli \cite{bb1}. Indeed, when $q=2$ the assumption $\omega\in L^q(\Omega\times [0,T])$ is automatically verified by the energy inequality \eqref{energy inequality}.

Theorem \ref{thm: main} provides a new characterisation for the control of vorticity under suitable alignment of the vortex structures in 3D incompressible fluids. Roughly speaking, it suggests a self-improvement property from the {\em average-in-time} bound for the (spatial) $L^q$ norm of $\omega$ to the {\em uniform-in-time} bound, provided that the vorticity does not change its directions too sharply wherever its magnitude is large.

Moreover, we remark that regularity conditions for vorticity have also been established under space-time integrability conditions on the vorticity magnitude. For example, Gruji\'{c}--Ruzmaikina \cite{gr} proved that for $\beta \in [1/q, 1]$ and $\int_0^T\big(\int|\omega(t,x)|^q\,\dd x\big)^{1/(q-1)}\,\dd t <\infty$, the $L^q$ norm of $\omega$ remains bounded as $t\map T^-$. The special case $q=2$ also coincides with the result by Beir\~{a}o da Veiga--Berselli \cite{bb1}.

\section{Preliminary Identities and Estimates}

In this section we summarise several identities and inequalities that shall be used in the subsequent development. 

First of all, we recall the singular integral representation of the rate-of-strain tensor $\scal$ in terms of $\omega$, which is crucial to the arguments in  Constantin--Fefferman \cite{cf}. Denoting by $\widehat{a}:=a/|a|$ for three-vectors $a\in\R^3$, there holds (Eq.\,(4) in \cite{cf}):
\begin{equation}
\scal (t,x) = \frac{3}{8\pi} {\rm p.v.}\int \bigg\{\frac{\xyhat \otimes \big(\xyhat \times \omega(t,x)\big) + \big(\xyhat \times \omega(t,x)\big)\otimes\xyhat}{|x-y|^3}\bigg\}\,\dd y.
\end{equation}
The symbol $p.v.$ denotes the principal value of the integral. Thus, the normalised vortex stretching term $\scal:(\what\otimes\what)$ can be expressed as follows:
\begin{equation}\label{vortex stretching term}
\scal:(\what\otimes\what)(t,x)=\frac{3}{4\pi} {\rm p.v.}\int\bigg\{ \frac{D\Big(\xyhat, \what(t,x), \what(t,x-y)\Big) |\omega(t,x)|}{|x-y|^3}\bigg\}\,\dd y,
\end{equation}
where
\begin{equation}
D(e_1,e_2,e_3) := (e_1 \cdot e_3) \,\det(e_1, e_2, e_3),
\end{equation}
and $e_i$ are three-vectors (column vectors) for $i=1,2,3$. As shown on pp.778--780 in \cite{cf}, the bound for the angle  $\varphi$ can be translated to a bound for the $D$ term:
\begin{lemma}\label{lemma: geometric}
Under the assumptions of Theorem \ref{thm: main},  we have
\begin{equation}
\Big|D\Big(\xyhat, \what(t,x), \what(t,x-y)\Big)\Big| \leq \frac{|x-y|^\beta}{\rho}.
\end{equation}
\end{lemma}

Next, the time-evolution of the $L^q$ norm of $\omega$ (for any $q\geq 1$) has been derived by Qian in \cite{qian}; see the proof of Lemma 2 therein: 
\begin{lemma}\label{lemma: qian}
Let $u$ be a weak solution to Eqs.\,\eqref{NS eqn}--\eqref{initial data}. Then, for $q\geq 1$, there holds
\begin{align}\label{qian}
&\frac{\dd}{\dd t}\bigg( \int |\omega(t,x)|^q\,\dd x\bigg) + \frac{4(q-1)}{q}\nu \bigg( \int \big|\na \big(|\omega(t,x)|^{q/2}\big)\big|^2\,\dd x\bigg)\,\nonumber\\
&\qquad \leq q \int |\omega(t,x)|^{q-2} \scal(t,x):\big(\omega(t,x) \otimes \omega(t,x)\big)\,\dd x.
\end{align} 
\end{lemma}

Finally, in Sect.\,3 we shall make crucial use of the Hardy--Littlewood--Sobolev interpolation inequality ({\it cf.} p106, Lieb--Loss \cite{ll}), with $n=3$, $\lambda=2+\delta$ and $f,h$ supported in $\Omega$:
\begin{lemma}\label{lemma: HLS}
Let $1<p,r<\infty$, $0<\lambda<n$ satisfy $1/p + \lambda/n + 1/r =2$. Let $f\in L^p(\R^n)$ and $h\in L^r(\R^n)$. Then there exists $C=C(n,\lambda,p)$ such that
\begin{equation}
\bigg| \int_{\R^n} \int_{\R^n} \frac{f(x)h(y)}{|x-y|^\lambda}\,\dd x\,\dd y \bigg| \leq C \|f\|_{L^p(\R^n)} \|h\|_{L^r(\R^n)}.
\end{equation}
\end{lemma}

\section{Proof of Theorem \ref{thm: main}}

Equipped with Lemmas \ref{lemma: geometric}--\ref{lemma: HLS} above, we are ready to prove Theorem \ref{thm: main}. 

As in Constantin--Fefferman \cite{cf}, let us decompose the  vorticity into ``big'' and ``small'' parts, with respect to the (large) constant $\Lambda > 0$ in Theorem \ref{thm: main}. To this end, taking $\chi \in C^\infty([0,\infty[)$, $0 \leq \chi\leq 1$, $\chi \equiv 1$ on $[0,1]$ and $\chi \equiv 0$ on $[2,\infty[$, we define
\begin{align}
\omega(t,x) &:= \osmall(t,x) + \obig(t,x)\nonumber\\
&=\chi\Big(\frac{|\omega(t,x)|}{\Lambda}\Big) \omega(t,x) + \bigg\{1-\chi\Big(\frac{|\omega(t,x)|}{\Lambda}\Big)\bigg\} \,\omega(t,x).
\end{align}
We also write 
\begin{equation}
\scal^{(i)} (t,x) = \frac{3}{8\pi} {\rm p.v.}\int \bigg\{\frac{\xyhat \otimes \big(\xyhat \times \omega^{(i)}(t,x)\big) + \big(\xyhat \times \omega^{(i)}(t,x)\big)\otimes\xyhat}{|x-y|^3}\bigg\}\,\dd y
\end{equation}
for $i\in\{<,>\}$, namely the corresponding singular integral with input $\omega^{(i)}$.

Now, in view of Lemma \ref{lemma: qian}, our goal is to estimate
\begin{align}
q \int |\omega(t,x)|^{q-2} \scal(t,x):\big(\omega(t,x) \otimes \omega(t,x)\big)\,\dd x =: q \int K(t,x)\,\dd x.
\end{align}
Following the notations in Constantin--Fefferman \cite{cf}, there holds
\begin{equation}\label{xx}
|K(t,x)| \leq |X(t,x)|+|Y(t,x)|+|Z(t,x)|,
\end{equation}
where
\begin{eqnarray*}
&&X(t,x):=  \sum_{(i,j)\neq (>,>)}|\omega(t,x)|^{q-2} \Big\{\scal(t,x): \big(\omega^{(i)}(t,x)\otimes\omega^{(j)}(t,x)\big)\Big\},\\
&&Y(t,x):= |\omega(t,x)|^{q-2} \Big\{\ssmall(t,x): \big(\obig(t,x)\otimes\obig(t,x)\big)\Big\},\\
&&Z(t,x):= |\omega(t,x)|^{q-2} \Big\{\sbig(t,x): \big(\obig(t,x)\otimes\obig(t,x)\big)\Big\}.
\end{eqnarray*}
We shall estimate these three terms in order in Subsections 3.1--3.3 below.

\subsection{The $X(t,x)$ Term}

To estimate $X(t,x)$, recall that $\omega^{(i)} \mapsto \scal^{(i)}$ is a Calderon--Zygmund singular integral; hence,  for some $C=C(r,\Omega)$ we have
\begin{equation}\label{cz}
\|\scal^{(i)}\|_{L^r} \leq C\|\omega^{(i)}\|_{L^r}\qquad \text{ for each } r \in ]1,\infty[ \text{ and } i \in \{<,>\}.
\end{equation}
As $|\omega^{(<)}| \leq \Lambda$, for $q>1$ we can bound by H\"{o}lder's inequality and Eq.\,\eqref{cz}:
\begin{align}\label{X}
\Big|\int X(t,x)\,\dd x\Big| &\leq \Lambda \int |\omega(t,x)|^{q-1} |\scal|\,\dd x\nonumber\\
&\leq \Lambda\Big\||\omega(t,\cdot)|^{q-1}\Big\|_{L^{\frac{q}{q-1}}} \|\scal\|_{L^q}\nonumber\\
&\leq C\Lambda\Big\||\omega(t,\cdot)|^{q-1}\Big\|_{L^{\frac{q}{q-1}}} \|\omega\|_{L^q} \leq C\Lambda \|\omega(t,\cdot)\|_{L^q}^q,
\end{align}
where $C=C(q,\Omega)$. 

\subsection{The $Y(t,x)$ Term}
For this purpose, let us denote by $\|\omega\|_{L^\infty(0,T;L^1(\Omega)} \leq \G$. Indeed, $\G$ is finite for any weak solution to Eqs.\,\eqref{NS eqn}--\eqref{initial data}
; see \cite{qian, cf}. Then we estimate
\begin{align}\label{Y}
\Big|\int Y(t,x)\,\dd x\Big| &\leq \int \big|\ssmall(t,x)\big| |\omega(t,x)|^q\,\dd x\nonumber\\
&\leq \bigg( \int \Big(|\omega(t,x)|^{q/2}\Big)^4\,\dd x\bigg)^{\frac{1}{2}} \bigg(\int \big|\ssmall(t,x)\big|^2\,\dd x \bigg)^{\frac{1}{2}}\nonumber\\
&\leq C \bigg(\int \Big|\na \big(|\omega(t,x)|^{q/2}\big)\Big|^2\,\dd x \bigg)^{\frac{3}{4}} \bigg(\int |\omega(t,x)|^q\,\dd x\bigg)^{\frac{1}{4}} \bigg(\int\big|\osmall(t,x)\big|^2\,\dd x\bigg)^{\frac{1}{2}}\nonumber\\
&\leq C\sqrt{\Lambda\G}\bigg(\int \Big|\na \big(|\omega(t,x)|^{q/2}\big)\Big|^2\,\dd x \bigg)^{\frac{3}{4}} \bigg(\int |\omega(t,x)|^q\,\dd x\bigg)^{\frac{1}{4}} \nonumber\\
&\leq \frac{2(q-1)}{q}\nu \bigg(\int \Big|\na \big(|\omega(t,x)|^{q/2}\big)\Big|^2\,\dd x \bigg) + C \nu^{-3}(\Lambda\G)^2 \bigg(\int |\omega(t,x)|^q\,\dd x\bigg).
\end{align}
In the second line we use the  Cauchy--Schwarz inequality; in the third line the Gagliardo--Nirenberg--Sobolev interpolation inequality (indeed, the special case known as the Ladyszhenskaya inequality) and Eq.\,\eqref{cz}; in the fourth line $|\osmall|\leq\Lambda$ and $\|\omega\|_{L^\infty(0,T;L^1(\Omega))} \leq \G$, and in the final line the Young's inequality $ab \leq a^4/4 + 3b^{4/3}/4$ for suitable $a,b\geq 0$.

\subsection{The $Z(t,x)$ Term}
$Z$ is the difficult term. To control it, we observe that $\wbig$ is the direction of vorticity on the region with large vorticity magnitude. Thanks to Eq.\,\eqref{vortex stretching term}, Lemma \ref{lemma: geometric} and the assumptions in Theorem \ref{thm: main}, we have
\begin{equation}
\Big|\int Z(t,x)\,\dd x\Big| \leq q J(t),
\end{equation}
where
\begin{eqnarray}
&&J(t) := \rho^{-1}\int |\omega(t,x)|^q I(t,x)\,\dd x,\label{J}\\
&&I(t,x):=\int \frac{|\omega(t,y)|}{|x-y|^\lambda}\,\dd y, \label{I}\\
&&\lambda = 3-\beta =: 2+\delta.
\end{eqnarray}

The bound for $J(t)$ is achieved by the lemma below. The parameters $\theta$, $\alpha$ involved therein will be carefully chosen later.
\begin{lemma}\label{lemma: J}
Suppose $\|\omega(t,\cdot)\|_{L^1(\Omega)}\leq\G$ for all $t\in [0,T]$; $q>1$. Then
\begin{align}
J(t) &\leq \frac{q-1}{q^2} \nu \bigg(\int \Big| \na \big(|\omega(t,x)|^{q/2}\big)\Big|^2\,\dd x\bigg)\nonumber\\
 &\qquad+ C \G^{\frac{\theta}{1-\alpha}} q^{-1}\nu^{-\frac{\alpha}{1-\alpha}} \rho^{-\frac{1}{\alpha}} \bigg(\int |\omega(t,x)|^q\,\dd x\bigg)^{1+\frac{1-\theta}{q(1-\alpha)}}.
\end{align}
The parameters $\theta$, $\alpha \in [0,1]$ depend on $q, \lambda$, and the constant $C$ depends on $\lambda, \Omega, q$ and $\alpha$.
\end{lemma}


\begin{proof}
The proof is divided into five steps.

\smallskip
{\bf 1.} By H\"{o}lder's inequality, there holds
\begin{equation}\label{estimate 1}
J(t) \leq \bigg(\int |\omega(t,x)|^{pq}\bigg)^{\frac{1}{p}} \bigg(\int |I(t,x)|^{p'}\bigg)^{\frac{1}{p'}} = \Big\| |\omega(t,\cdot)|^{\frac{q}{2}}\Big\|^2_{L^{2p}} \,\|I(t,\cdot)\|_{L^{p'}}.
\end{equation}
We write $p' = \frac{p}{p-1}$ for the conjugate exponent of $p$. For the moment we require no condition on the index $p$ more than $p\in]1,\infty[$.

\smallskip
{\bf 2.} The Gagliardo--Nirenberg--Sobolev interpolation inequality ({\it cf.} Nirenberg \cite{n}) yields
\begin{equation}\label{estimate 2}
 \Big\| |\omega(t,\cdot)|^{\frac{q}{2}}\Big\|^2_{L^{2p}} \leq C_1 \bigg(\int \Big| \na \big(|\omega(t,x)|^{q/2}\big)\Big|^2\,\dd x\bigg)^\alpha \bigg(\int |\omega(x)|^q\,\dd x\bigg)^{1-\alpha},
\end{equation}
where $\alpha\in ]0,1[$ is chosen such that
\begin{equation}\label{p, alpha}
p=\frac{3}{3-2\alpha}.
\end{equation}
We notice that $p\in ]1,3[$; $C_1$ depends only on $q, \Omega$.

\smallskip
{\bf 3.} For the $I$ term in Eq.\,\eqref{estimate 1}, we apply the Hardy--Littlewood--Sobolev interpolation inequality (Lemma \ref{lemma: HLS} plus an elementary duality argument) to find
\begin{equation}\label{estimate 3}
\|I(t,\cdot)\|_{L^{p'}} \leq C_2 \|\omega(t,\cdot)\|_{L^\sigma},
\end{equation}
where $C_2=C(p,\lambda,\Omega)$. The indices satisfy
\begin{equation}\label{sigma, p, lambda}
\frac{1}{\sigma} + \frac{1}{p} + \frac{\lambda}{3} =2,
\end{equation}
with $1 < \sigma, p <\infty$ and $0 < \lambda < 3$. For our purpose we shall specialise to $\lambda \in [2,3[$, hence $\delta = \lambda-2 \in [0,1[$, as well as $\sigma \in [1,q]$.

\smallskip
{\bf 4.} Now let us put together Eqs.\,\eqref{estimate 1}\eqref{estimate 2}\eqref{estimate 3} and apply Young's inequality $ab \leq \alpha a^{\frac{1}{\alpha}} + (1-\alpha) b^{\frac{1}{1-\alpha}}$ for suitable $a,b\geq 0$ (recall that $\alpha \in ]0,1[$). This gives us
\begin{align}\label{estimate 4}
J(t) &\leq \frac{q-1}{4q^2} \nu \bigg(\int \Big| \na \big(|\omega(t,x)|^{q/2}\big)\Big|^2\,\dd x\bigg) \nonumber\\
&\qquad + C_3 q^{-1}\nu^{-\frac{\alpha}{1-\alpha}} \rho^{-\frac{1}{\alpha}} \bigg(\int|\omega(t,x)|^q\,\dd x\bigg) \bigg(\int |\omega(t,x)|^\sigma\,\dd x\bigg)^{\frac{1}{\sigma (1-\alpha)}},
\end{align}
where the constant $C_3$ above depends on $q, p, \lambda, \Omega$ (note that $\alpha$ is determined by $p$).

\smallskip
{\bf 5.} Finally, for $1\leq \sigma \leq q$ we have the interpolation inequality for Lebesgue spaces:
\begin{equation}
\bigg(\int |\omega(t,x)|^\sigma\,\dd x\bigg)^{\frac{1}{\sigma (1-\alpha)}} \leq \bigg( \int |\omega(t,x)|\,\dd x\bigg)^{\frac{\theta}{1-\alpha}} \bigg(\int|\omega(t,x)|^q\,\dd x\bigg)^{\frac{1-\theta}{q(1-\alpha)}},
\end{equation}
with $\theta \in [0,1]$ determined by
\begin{equation}\label{sigma, theta q}
\frac{1}{\sigma} = \frac{\theta}{1} + \frac{1-\theta}{q}.
\end{equation}
In view of Steps 1--5, the proof is complete.   \end{proof}

\subsection{A Condition on the Indices $\theta, \alpha$}

Now, let us single out the following condition
\begin{equation*}
1-\theta \leq q (1-\alpha) \tag{$\clubsuit$}
\end{equation*}
on the parameters $\theta$, $\alpha$ in Lemma \ref{lemma: J}. We then have the following
\begin{lemma}\label{lemma: completion of proof}
Assume $(\clubsuit)$ for the parameters $\theta$, $\alpha$ in Lemma \ref{lemma: J}. Then, under the assumptions of Theorem \ref{thm: main}, we have $|\omega|^{q/2} \in L^\infty(0,T;L^2(\Omega)) \cap L^2(0,T;H^1(\Omega))$.
\end{lemma}

\begin{proof}
	Collecting the estimates in Sects.\,3.1--3.3 (in particular,  Eqs.\,\eqref{xx}\eqref{X}\eqref{Y} and Lemma \ref{lemma: J}), we can reduce the estimate in Lemma \ref{lemma: qian} to the following:
	\begin{align}\label{Q}
	\frac{\dd \mathcal{Q}}{\dd t}  + \frac{q-1}{q}\nu \bigg(\int \big|\na (|\omega(t,x)|^{q/2}\big|^2\,\dd x\bigg) \leq C\bigg\{\Lambda + \nu^{-3}\Lambda^2 \G^2 + \G^{\frac{\theta}{1-\alpha}}\nu^{-\frac{\alpha}{1-\alpha}}\rho^{-\frac{1}{\alpha}}\mathcal{Q}^\gamma\bigg\}\mathcal{Q},
	\end{align}
where $C$ depends only on $q$ and $\Omega$, $\gamma:=\frac{1-\theta}{q(1-\alpha)}<1$ by $(\clubsuit)$, and
\begin{equation}
\mathcal{Q}(t):=\int|\omega(t,x)|^q\,\dd x.
\end{equation}
Now, let us invoke the assumption $\omega\in L^q(\Omega\times [0,T])$ in Theorem \ref{thm: main} to get
\begin{equation}
\int_0^T \mathcal{Q}^\gamma \,\dd t \leq T^{1-\gamma} \|\omega\|_{L^q(\Omega\times [0,T])}^{q\gamma} < \infty.
\end{equation}
Therefore, by the ordinary differential inequality (neglecting the $\int |\na (|\omega(t,x)|^{q/2})|^2\,\dd x$ term that has the favourable sign), 
\begin{equation}
\mathcal{Q}(t) \leq \mathcal{Q}_0 \,\exp\bigg\{C(\Lambda+\nu^{-3}\Lambda^2\G^2)T + CT^{1-\gamma}\|\omega\|_{L^q(\Omega\times [0,T])}^{q\gamma} \G^{\frac{\theta}{1-\alpha}}\nu^{-\frac{\alpha}{1-\alpha}}\rho^{-\frac{1}{\alpha}}\bigg\}.
\end{equation}
This proves $|\omega|^{q/2} \in L^\infty(0,T;L^2(\Omega))$. The other half of the conclusion follows immediately from Eq.\,\eqref{Q}; so the proof is complete.  \end{proof}

\subsection{Completion of the Proof of Theorem \ref{thm: main}} Finally, we observe that Theorem \ref{thm: main}  follows directly from the conclusion of Lemma \ref{lemma: completion of proof}. Therefore, to complete the proof, it remains to specify the indices $\theta$, $\alpha$ in Lemma \ref{lemma: J} that satisfy $(\clubsuit)$. Here enters the restriction of the range of $q$ to $]5/3,\infty[$, as well as the choice of $\theta$ and $\alpha$, which are dependent on $q$. In effect, the choice of $\theta, \alpha$ amounts to the specification of the parameter $p$ in the proof of Lemma \ref{lemma: J}. To this end, we shall keep track in detail of the range of parameters for all the inequalities involved in the proof of Lemma \ref{lemma: J}:
\begin{proof}[Proof of Theorem \ref{thm: main}]
Recall that $p \in ]1,3[$ in the proof of Lemma \ref{lemma: J}; by Eq.\,\eqref{sigma, p, lambda} and $\lambda=2+\delta$,
\begin{equation}
p = \frac{3\sigma}{(4-\delta)\sigma - 3}.
\end{equation}
Here $\sigma \in [1,q]$ due to  Step 5 in the proof of Lemma \ref{lemma: J}, hence $p\in [\frac{3q}{(4-\delta)q-3}, \frac{3}{1-\delta}]$. Moreover, the right endpoint $\frac{3}{1-\delta}$ is no less than $3$. We thus have
\begin{equation}\label{range for p}
p \,\in \big]1,3\big[\, \bigcap \, \Big[\frac{3q}{(4-\delta)q-3}, 3\Big[.
\end{equation}
This condition is non-vacuous, since $q>5/3$ implies $\frac{3q}{(4-\delta)q-3}<3$ for $\delta \in [0,1[$. In the sequel we shall use it to derive more stringent conditions on $\delta$.

Now let us express all the other constants --- $\sigma$, $\theta$ and $\alpha$ --- in terms of $p$, and then match the condition ($\clubsuit$). Indeed, from Eq.\,\eqref{sigma, p, lambda} we get
\begin{equation}\label{sigma}
\sigma = \frac{3p}{(4-\delta)p -3}.
\end{equation}
Together with Eq.\,\eqref{sigma, theta q}, it leads to
\begin{equation}\label{theta}
\theta = \frac{(4-\delta)pq - 3(p+q)}{3p (q-1)}.
\end{equation}
Also, Eq.\,\eqref{p, alpha} can be written as
\begin{equation}\label{alpha}
\alpha = \frac{3}{2}\Big(1-\frac{1}{p}\Big).
\end{equation}
Thus, substituting in Eqs.\,\eqref{sigma}\eqref{theta} and \eqref{alpha}, an elementary computation shows that ($\clubsuit$) is equivalent to $3q(3-p) + (5-2\delta)p - 15 \geq 0$. This gives us an upper bound for $\delta$:
\begin{equation}\label{upper bound for delta}
\delta \leq \frac{(3q-5)(3-p)}{2p} =: \mathcal{U}(p,q).
\end{equation}
Notice that for $q > 5/3$ (by assumption) and $p < 3$ (by Eq.\,\eqref{range for p}), our condition \eqref{upper bound for delta} allows for non-trivial $\delta \in [0,1[$; in addition, $p\mapsto\mathcal{U}(p,q)$ is decreasing on $]1,3[$.

To conclude the proof, let us observe
\begin{equation}\label{obs}
\frac{3q}{(4-\delta)q -3} < \, (=, \, >)\, 1 \,\,\Longleftrightarrow \,\, \delta < \, (=, \, >)\, 1-\frac{3}{q}.
\end{equation}
Therefore, to ensure $\delta \geq 0$, for $q \in ]5/3,3]$ we must require $\frac{3q}{(4-\delta)q -3} >1$; in this case, the choice of the index $p$ is restricted to $[\frac{3q}{(4-\delta)q-3},3[$, in view of Eq.\,\eqref{range for p}. On the other hand, in $q>3$, both $\delta <1-3/q$ and $\delta \geq 1-3/q$ is allowed.

{\noindent}
\underline{Case 1}: $q \in ]5/3,3]$. As discussed above, one may choose $p$ arbitrarily in $[\frac{3q}{(4-\delta)q-3},3[$. To maximise $\mathcal{U}(p,q)$ in Eq.\,\eqref{upper bound for delta}, let us take $p=\frac{3q}{(4-\delta)q-3}$; so
\begin{equation*}
\delta \leq (3q-5) \bigg(\frac{3-\frac{3q}{(4-\delta)q-3}}{\frac{6q}{(4-\delta)q-3}}\bigg) = \frac{\big[(3-\delta)q -3\big](3q-5)}{2q},
\end{equation*}
which is equivalent to $\delta \leq 3-\frac{5}{q}$. Together with $\delta \in [0,1[$, we get 
\begin{equation}
\delta \in \Big[ 0, \min\Big\{1, 3-\frac{5}{q}\Big\}\Big] \qquad \text{ for } q \in \big]\frac{5}{3}, 3\big].
\end{equation}

{\noindent}
\underline{Case 2}: $q \in ]3,\infty[$. In this case, if $\delta \leq 1-3/q$, namely $\frac{3q}{(4-\delta)q-3}\leq 1$ by Eq.\,\eqref{obs}, in light of Eq.\,\eqref{range for p} we can take any $p \in ]1,3[$. To maximise $\mathcal{U}(p,q)$ we choose $p\map 1^+$; then Eq.\,\eqref{upper bound for delta} gives us $\delta < 3q-5$, which is automatically true for $\delta \in [0,1[$. On the other hand, if $\delta > 1-3/q$, {\it i.e.}, $\frac{3q}{(4-\delta)q-3}>1$, then by the computation in Case 1 above, we have $\delta \leq 3-5/q$. In summary, 
\begin{equation}
\delta \in \Big[0,1-\frac{3}{q}\Big] \,\bigcup \,\Big[1-\frac{3}{q}, \min\Big\{1, 3-\frac{5}{q}\Big\}\Big] \equiv \Big[ 0, 1\Big] \qquad \text{ for } q \in \big]3,\infty\big[.
\end{equation}
Putting together Cases 1--2 and using $3-\beta=\lambda+2$, we now complete the proof.  \end{proof}

\bigskip
\noindent
{\bf Acknowledgement}.
The author is grateful to Prof.\,Zhongmin Qian for many insightful discussions and generous sharing of ideas, to Prof.\,Gui-Qiang G. Chen for his lasting support, and to Prof.\,Zoran Gruji\'{c} for communicating with us the paper \cite{gr}. Part of this work has been done during Siran Li's stay as a CRM--ISM postdoctoral fellow at  the Centre de Recherches Math\'{e}matiques, Universit\'{e} de Montr\'{e}al and the Institut des Sciences Math\'{e}matiques. The author would like to thank these institutions for their hospitality.

\end{document}